\documentclass[11pt,a4paper]{amsart}
\usepackage[T1]{fontenc}
\usepackage[utf8]{inputenc}
\usepackage{lmodern}
\usepackage{amsmath,amssymb,amsthm,mathtools}
\usepackage[a4paper,left=27mm,right=27mm,top=26mm,bottom=27mm]{geometry}
\usepackage{microtype}
\usepackage{needspace}
\usepackage{verbatim}
\usepackage{xcolor}
\usepackage[colorlinks=true,linkcolor=blue!45!black,citecolor=blue!45!black,urlcolor=blue!45!black]{hyperref}
\newtheorem{theorem}{Theorem}[section]
\newtheorem{lemma}[theorem]{Lemma}
\newtheorem{corollary}[theorem]{Corollary}

\theoremstyle{definition}

\newcommand{\E}{\mathbb E}

\newcommand{\dd}{\mathrm d}
\newcommand{\e}{\mathrm e}

\title{Another proof that the two color bipartite Ramsey number is $O(2^t)$.}

\author[Y.~Person]{Yury Person}
\address{Institut f\"ur Mathematik, Technische Universit\"at Ilmenau, 98684 Ilmenau, Germany} 
\email{yury.person@tu-ilmenau.de}

\date{29 September 2026}
\begin{document}
\begin{abstract}
For positive integers $t$ and $q$ let $b_q(t)$ be the smallest integer $n$ so that any coloring of the edges of the complete bipartite graph $K_{n,n}$ with $q$ colors yields a monochromatic copy of $K_{t,t}$. 
We give an independent proof that $b_2(t)\le 128\cdot 2^t$ for every positive integer $t$. More generally, for
$0< p\le 1/2$, every bipartite graph of edge density at least $p$, with both classes of size at least $128\cdot p^{-t}$,
contains $K_{t,t}$. This gives $b_q(t)\le 128\cdot q^t$ for every integer $q\ge 2$ and a uniform consequence for
Zarankiewicz numbers.
\end{abstract}

\maketitle

\noindent\textbf{AI disclosure.}
The proof was found by GPT-6 Astra after some conversations. I edited the presentation and verified all proofs.

\section{Introduction}

The classical Ramsey number $r(s,t)$ is the least integer $n$ such that every red--blue coloring of the edges of $K_n$ contains a red $K_s$ or a blue $K_t$. It was first studied by Ramsey in~\cite{Ramsey1930}. Recently it has seen some spectacular progress on improvements both for the symmetric and asymmetric regimes and for many colors, see this survey by Morris~\cite{Morris2026}.

Beineke and Schwenk~\cite{BeinekeSchwenk1976} introduced the corresponding bipartite form of a Ramsey problem. For positive integers $s$, $t$ let $b(s,t)$ be the smallest integer $n$ so that any coloring of the edges of the complete bipartite graph $K_{n,n}$ with two colors, say red and blue, yields a monochromatic copy of $K_{s,t}$, which we will often refer to as a biclique.  In~\cite{BeinekeSchwenk1976} such bipartite Ramsey numbers are determined for some small constant values of one of the classes and the conjecture $b(s,t)=2^s(t-1)+1$, for $s\le t$, is made. 

This question is closely related to the Zarankiewicz extremal problem. 
For vertex classes $V_1$, $V_2$ of sizes $n_1$, $n_2$, let $z(n_1,n_2;s,t)$ be the maximum number of edges in a bipartite graph containing no $K_{s,t}$ with its $s$ vertices in the first class and its $t$ vertices in the second class. K\H{o}v\'ari, S\'os and Tur\'an~\cite{KST1954} proved the explicit bound
\begin{equation}\label{eq:KST}
 z(n,n;t,t)< 1+ tn+(t-1)^{1/t} n^{\frac{2t-1}{t}}.%\le (t-1)^{1/t}n^{2-1/t}+(t-1)n.
\end{equation}
The argument of~\cite{KST1954} is based on counting common neighbors and convexity and yields the following general explicit bound
\begin{equation}\label{eq:kst-rect}
 z(n_1,n_2;s,t)
 \le (t-1)^{1/s}n_1n_2^{1-1/s}+(s-1)n_2.
\end{equation}

The connection of $b(s,t)$ to the Zarankiewicz extremal problem was developed by Irving~\cite{Irving1978} who disproved the conjecture from~\cite{BeinekeSchwenk1976} and whose result yields in particular $b(t,t)<2^{t-1}t$ for $t\ge 21$.

Thomason~\cite{Thomason1982} proved the asymmetric upper bound $b(s,t)\le 2^s(t-1)+1$. 

For $s=t$ we simply write $b(t)$ instead of $b(t,t)$. And, more generally, if $q$ colors are used, we write $b_q(t)$ to denote the least integer $n$ such that every $q$-coloring of $K_{n,n}$ contains a monochromatic $K_{t,t}$. 

As for the lower bound, Hattingh and Henning~\cite{HattinghHenning1998} used Lov\'asz local lemma to prove
\begin{equation}\label{eq:lower}
 b_2(t)\ge
 \left(\frac{\sqrt2}{\e}-o(1)\right)t2^{t/2}.
\end{equation}

Conlon~\cite{Conlon2008},  building on the work of F\"uredi~\cite{Fueredi1996},  improved the upper bound to 
\begin{equation}\label{eq:conlon}
 b_2(t)\le(1+o(1))\,2^{t+1}\log_2t.
\end{equation}

Conlon also proved that, for $t\ge s$, 
$b(s,t)\le (1+o_{s\to\infty}(1))2^s(t-s+2\log_2s)$.

A very recent preprint of Mubayi~\cite{Mubayi2026} removes the $\log t$ factor in Conlon's bound: $b_2(t)\le 2^{t+81}$. More generally,  Mubayi~\cite[Corollary~7.1]{Mubayi2026} shows that a bipartite graph of edge density $p\in(0,1)$, with $n$ vertices in each class, contains $K_{t,t}$, whenever $n\ge C_p p^{-t}$, where $C_p=2^{41}[p(1-p)]^{-20}$.

Our main result is a density criterion with an \emph{absolute} constant in place of the factor of order $t$ supplied by~\eqref{eq:KST}. Our argument was obtained independently of Mubayi's result~\cite{Mubayi2026} and, in fact, uses a different method. Mubayi~\cite{Mubayi2026} also works from a density assumption, but selects one adjacent pair at a time, restricting both classes to the appropriate neighbourhoods. Our argument can be viewed as a generalization of Conlon's approach~\cite{Conlon2008}, by  alternating the roles of the two vertex classes and reducing the remaining target sizes logarithmically.  No assumption about a second color is involved. 

\begin{theorem}[Density bound]\label{thm:density}
Let $t\ge1$ be an integer and let $0<p\le1/2$. Let $G$ be a bipartite graph with vertex classes $V_1,V_2$ of sizes $n_1,n_2$ and with edge density $\dd(V_1,V_2):=\frac{e(G)}{n_1n_2}$ at least $p$. If
\[
 \min\{n_1,n_2\}\ge128p^{-t},
\]
then $G$ contains $K_{t,t}$.
\end{theorem}
We remark that Mubayi's general density consequence has a constant $C_p$ depending  on the density $p$, whereas Theorem~\ref{thm:density} uses the same constant $128$ for all $0<p\le1/2$.

Inequalities comparing integer class sizes with real thresholds are understood literally; no rounding of the class sizes is needed in the proof.

As a corollary this immediately implies the bound claimed in the abstract.
\begin{corollary}\label{cor:two}
For every integer $t\ge1$,
\[
 b_2(t)\le128\,2^t.
\]
\end{corollary}
\begin{proof}
In a red--blue coloring of $K_{128\,2^t,128\,2^t}$, one color has density at least $1/2$. Apply Theorem~\ref{thm:density} to that color.
\end{proof}

\subsection{Organization of the paper}
In  Section~\ref{sec:basics} we record a common-neighbor counting lemma and a lower-tail moment estimate.
In Section~\ref{sec:proof_overview} we describe the proof strategy and its combinatorial invariant. In Section~\ref{sec:mainthm} we prove Theorem~\ref{thm:density}. Section~\ref{sec:closing} gives further consequences.

\section{Auxiliary lemmas}\label{sec:basics}

For a bipartite graph $G$ and a set $U$ contained in one vertex class of $G$, we write $N(U)$ for its common neighborhood in the other class. For a real number $x$, write $x_+=\max\{x,0\}$. The following lemma is a variant of the counting argument from~\cite{KST1954}. 

\begin{lemma}[Common-neighbor averaging]\label{lem:count}
Let $G$ be a bipartite graph with  vertex classes $V_1$ and  $V_2$ with corresponding sizes $n_1$ and $n_2$ and edge density $p$. If $1\le a\le n_1$ is an integer, then
\begin{equation}\label{eq:average}
 \frac{1}{\binom{n_1}{a}}
 \sum_{\substack{U\subseteq V_1\\|U|=a}}|N(U)|
 \ge n_2\left(p-\frac{a-1}{n_1}\right)_+^a.
\end{equation}
 If $G$ contains no $K_{r,s}$ with its $r$ vertices in $V_1$, where $1\le r\le n_1$, then
\begin{equation}\label{eq:forbidden-stars}
 \sum_{y\in V_2}\binom{\deg(y)}r
 \le(s-1)\binom{n_1}r.
\end{equation}
Both statements hold with the classes interchanged.
\end{lemma}
\begin{proof}
Count pairs $(U,y)$ with $U\subseteq V_1$, $|U|=a$, and $y$ adjacent to every vertex of $U$. This gives the following identity
\[
 \sum_{\substack{U\subseteq V_1\\|U|=a}}|N(U)|
 =\sum_{y\in V_2}\binom{\deg(y)}{a}.
\]
 For every nonnegative integer $z$ the following bounds hold
\[
 a!\binom za\ge(z-a+1)_+^a,
 \qquad
 a!\binom{n_1}a\le n_1^a.
\]
Since $x\mapsto(x-a+1)_+^a$ is convex, the above bounds with Jensen's inequality provide
\[
\frac{1}{\binom{n_1}{a}}
 \sum_{\substack{U\subseteq V_1\\|U|=a}}|N(U)|
 = \frac{\sum_{y\in V_2}\binom{\deg(y)}a}{\binom{n_1}a}
 \ge \frac{n_2}{n_1^a}
 \left(\sum_{y\in V_2}\frac{\deg(y)-a+1}{n_2}\right)_+^a.
\]
Since the edge density is $p$, the average vertex degree from $V_2$ into $V_1$ is $pn_1$, proving \eqref{eq:average}. Finally, if the specified $K_{r,s}$ is absent, every $r$-set in $V_1$ has at most $s-1$ common neighbors. The same pair count gives \eqref{eq:forbidden-stars}.
\end{proof}

We remark that the bound~\eqref{eq:kst-rect} on the Zarankiewicz number $z(n_1,n_2;s,t)$ mentioned in the introduction follows immediately by applying Lemma~\ref{lem:count} above with $a=s$ to a $K_{s,t}$-free graph.

The next lemma turns a bound on a high moment into control of degrees below their intended scale.

\begin{lemma}[A lower-tail moment estimate]\label{lem:moment}
Let $X$ be a nonnegative random variable, let $r\ge1$ be an integer, and suppose that $\E X^r\le C$. Then
\begin{equation}\label{eq:moment}
 \E(1-X)_+\le(1+C)^{1/r}-\E X.
\end{equation}
In the setting of \eqref{eq:forbidden-stars}, let $b>0$, write $n_2=\beta b^r$, choose $y\in V_2$ uniformly, and set
\begin{equation}\label{eq:X}
 X=\frac{b(\deg(y)-r+1)_+}{n_1}.
\end{equation}
Then
\begin{equation}\label{eq:exact-moment}
 \E X^r\le\frac{s-1}{\beta}.
\end{equation}
\end{lemma}
\begin{proof}
Put $Y=\max\{X,1\}$. Pointwise $Y^r\le1+X^r$, so Jensen's inequality gives 
$(\E Y)^r\le \E Y^r$ and hence
\[
 \E(X-1)_+=\E Y-1\le (\E Y^r)^{1/r}-1\le (1+\E X^r)^{1/r}-1\le (1+C)^{1/r}-1.
\]
Subtracting $\E(X-1)=\E X-1$ and using $x_+-x=(-x)_+$  proves~\eqref{eq:moment}.

For \eqref{eq:exact-moment}, use Lemma~\ref{lem:count} and the two factorial estimates from its proof:
\begin{align*}
 \E X^r=\E \frac{b^r(\deg(y)-r+1)_+^r}{n_1^r} \le 
  \frac{b^r}{n_1^r}\E\, r!\binom{\deg(y)}{r}
 &=\frac{b^r r!}{n_2n_1^r}
       \sum_{y\in V_2}\binom{\deg(y)}r\\
 &\overset{\eqref{eq:forbidden-stars}}{\le}\frac{b^rr!(s-1)\binom{n_1}r}{n_2n_1^r}
 \le\frac{s-1}{\beta},
\end{align*}
where in the last step we used $n_2=\beta b^r$.
\end{proof}

\section{Proof overview}\label{sec:proof_overview}

We first describe the construction in the forward direction. The formal proof in Section~\ref{sec:mainthm} uses the shorter equivalent argument by descending forbidden targets.

A book consists of a chosen set, its \emph{spine}, together with all its common neighbors in the opposite class. After choosing a large spine in the second class, we need only a much smaller number of additional vertices there. The remaining problem is therefore unbalanced. We then choose a large spine in the first class, interchange the roles of the classes, and repeat. The numbers of vertices still required decrease logarithmically.

Here is the combinatorial invariant underlying this description. The chosen cores $C_1,C_2$ and available classes $V_1,V_2$ satisfy
\begin{equation}\label{eq:invariant}
 \begin{gathered}
 |C_1|=t-r,\qquad |C_2|=t-s,\qquad C_i\cap V_i=\varnothing,\\
 C_1\text{ is complete to }C_2\cup V_2,
 \qquad C_2\text{ is complete to }V_1.
 \end{gathered}
\end{equation}
For each $i$, $C_i$ and $V_i$ lie in the same original host class; the two such classes may be relabelled after a round. A $K_{r,s}$ between the available classes completes the desired $K_{t,t}$.

For a real base $b\ge2$, put
\begin{equation}\label{eq:g}
 g_b(r)=1+\lceil3\log_b r\rceil.
\end{equation}
The typical residual target has $s=g_b(r)$, and the class sizes are measured as
\[
 n_1=\alpha b^s,\qquad n_2=\beta b^r.
\]
For fixed $b$, the reference scale $b^s$ is polynomial in $r$, whereas $b^r$ is exponential. The actual class sizes can be larger: only lower bounds on $\alpha,\beta$ are maintained. The numerical argument keeps these normalized sizes bounded below by an absolute constant and the density at least $(1-2/r^2)/b$.

Put $k=g_b(s)$. After deleting a small proportion of vertices of insufficient degree from the second class, choose a spine $U\subseteq V_1$ of size $r-k$ and retain its common neighborhood $W\subseteq V_2$. Append $U$ to $C_1$ and exchange the classes. The updates are
\[
 (C_1,C_2;V_1,V_2)\longmapsto
 (C_2,C_1\cup U;W,V_1\setminus U),
 \qquad (r,s)\longmapsto(s,k).
\]
They preserve \eqref{eq:invariant}: $U$ is complete to $W$ by definition, and all other required edges follow from the old invariant. The pruning is justified by Lemma~\ref{lem:moment}, under the assumption that the current target is absent. The losses in $\alpha,\beta$ have a summable upper bound indexed by the successive smaller targets $s$. This is why the constant factor in the initial class size does not grow with $t$.

To initialize using density alone, take $w=2^{-1/t}$ and $b=(pw)^{-1}$. In the first common-neighbor count, subtract the entire possible contribution of vertices of degree below $pw n_2$. Enough high-degree common neighbors remain to start the recursion, and their degree condition survives removal of the initial spine. The choice $w^t=1/2$ balances the two normalized supplies. Thus the initialization, like the recursion, uses only the density of a single graph.

\section{Proof of the density bound}\label{sec:mainthm}

We prove a recursive density statement first. Its proof is a contradiction argument: if the desired unbalanced biclique were absent, the same would be true of successively smaller targets, with controlled losses. The process ends where a single application of Lemma~\ref{lem:count} already forces the target.

\begin{lemma}[The recursive density statement]\label{lem:core}
Let $b\ge2$ be real, let $r\ge16$ be an integer, and put $s=g_b(r)$. Suppose a bipartite graph $G$ with vertex classes $V_1$, $V_2$ has sizes
\[
 n_1=\alpha b^s,\qquad n_2=\beta b^r,
 \qquad  \text{ where } \alpha,\beta\ge62,
\]
and edge density $p$ at least $(1-2/r^2)/b$. Then $G$ contains $K_{r,s}$ with its $r$ vertices in the first class.
\end{lemma}

\begin{proof}
Assume that the specified $K_{r,s}$ is absent in $G$. At every stage of the descending argument, we will have
\begin{equation}\label{eq:maintain}
 r\ge16,\qquad s=g_b(r),\qquad
 \alpha,\beta\ge56,\qquad
 p\ge\frac{1-2/r^2}{b}.
\end{equation}
The initial lower bound $62$ leaves room for losses. We verify the maintained lower bound $56$ by estimating their total, rather than assuming that it is preserved by an individual step.

Several elementary bounds will be used repeatedly. From the definition of $g_b$,
\begin{equation}\label{eq:basic-g}
 b^s\ge br^3,
 \qquad
 s\le2+3\log_2r\le r\quad(r\ge16).
\end{equation}
Thus
\begin{equation}\label{eq:n-lower}
 n_1\ge56br^3.
\end{equation}
The extra $1$ in the definition of $g_b(r)$ in~\eqref{eq:g} supplies the factor $b$ in this estimate, which is needed uniformly as $b$ becomes large.

\medskip

\noindent\textbf{The terminal case $s\le56$.}
Lemma~\ref{lem:count}, with its slightly weaker form $p-r/n_1$ in place of $p-(r-1)/n_1$, gives an average common-neighbor count over $r$-sets of at least
\begin{align}
 n_2\left(\frac{1-2/r^2}{b}-\frac r{n_1}\right)^r
 &=\beta\left(1-\frac2{r^2}-\frac{br}{n_1}\right)^r\notag\\
 &\ge\beta\left(1-\frac2r-\frac{br^2}{n_1}\right)
 \ge\beta\left(1-\frac3r\right).
 \label{eq:terminal}
\end{align}
In the estimates we used Bernoulli's inequality  $(1-x)^r\ge1-rx$ for $0\le x\le1$ and $br^2/n_1\le1/(56r)\le1/r$.

If $r\ge256$, the final expression in \eqref{eq:terminal} is at least
\[
 56(1-3/256)>55\ge s-1.
\]
If $16\le r<256$, then \eqref{eq:basic-g} gives $s\le25$, while the same expression is at least
\[
 56(1-3/16)=45.5>24\ge s-1.
\]
In either case some $r$-set has more than $s-1$ common neighbors. Since that number is an integer, it has at least $s$ common neighbors, a contradiction.

\medskip

\noindent\textbf{Preparing a step when $s\ge57$.}
Put $k=g_b(s)$ and recall the definition $g_b(s):=1+\lceil 3\log_b s\rceil$. The ceiling relation in \eqref{eq:g} implies
\begin{equation}\label{eq:separation}
 r>b^{(s-2)/3}\ge2^{(s-2)/3}>2s,
 \qquad
 k\le2+3\log_2s<s.
\end{equation}
For the strict lower bound on $r$, use $\lceil3\log_b r\rceil=s-1$, whence $3\log_b r>s-2$. The inequality $2^{(s-2)/3}>2s$ holds at $s=57$ and its left-to-right ratio increases thereafter. The last inequality follows in the same way as \eqref{eq:basic-g}.

Choose $y\in V_2$ uniformly and define $X$ by \eqref{eq:X}. Lemma~\ref{lem:moment} and the absence of $K_{r,s}$ give
\begin{equation}\label{eq:core-moment}
 \E X^r\le\frac{s-1}{\beta}\le\frac{s-1}{56}.
\end{equation}
Also, $(z-r+1)_+\ge z-r$ for every $z\ge0$, so the density hypothesis yields
\begin{equation}\label{eq:core-mean}
 \E X\ge bp-\frac{br}{n_1}
 \ge1-\frac2{r^2}-\frac{br}{n_1}.
\end{equation}
Write
\[
 L_s=\log\left(1+\frac{s-1}{56}\right).
\]
All logarithms without a subscript are natural. Since $0\le L_s/r<1$, the inequality $\e^u-1\le2u$ for $0\le u\le1$, together with Lemma~\ref{lem:moment}, gives
\begin{equation}\label{eq:low-tail}
 \E(1-X)_+
 \le\frac{2L_s}{r}+\frac2{r^2}+\frac{br}{n_1}.
\end{equation}
Here $L_s<s<r$, which verifies the stated range of $u$.

Call $y\in V_2$ \emph{bad} if
\[
 \frac{\deg(y)}{n_1}<\frac{1-1/s^2}{b}.
\]
For a bad vertex, $X\le b\deg(y)/n_1<1-1/s^2$. Consequently, if $\theta$ is the fraction of bad vertices, \eqref{eq:low-tail} implies
\begin{equation}\label{eq:theta}
 \theta\le
 \frac{2s^2L_s}{r}+\frac{2s^2}{r^2}
       +\frac{bs^2r}{n_1}.
\end{equation}
Define
\[
 \lambda_s=\theta+\frac4r,
 \qquad
 J(s)=\bigl(2s^2L_s+5\bigr)2^{-(s-2)/3}.
\]
By \eqref{eq:n-lower},
\begin{equation}\label{eq:lambda}
 r\lambda_s
 \le2s^2L_s+\left(2+\frac1{56}\right)\frac{s^2}{r}+4
 <2s^2L_s+5=J(s)2^{(s-2)/3},
 \qquad \lambda_s<J(s).
\end{equation}
For the strict inequality, \eqref{eq:separation} gives
$s^2/r<s^2 2^{-(s-2)/3}$. This last function decreases for $s\ge57$; at $s=57$, its product with $2+1/56$ is less than $1$.

We now record the estimates which make the losses summable. They also show immediately that $\theta<\lambda_s<1$, so at least one good vertex remains. Put
\[
 F(x)=2x^2\log\frac{x+55}{56}+5\qquad(x\ge57).
\]
The inequality
\[
 \log\frac{x+55}{56}\ge\frac{x}{x+55}
\]
holds at $x=57$, since $\log2>57/112$, and the derivative of the difference is $x/(x+55)^2>0$. It follows that $F'(x)/F(x)\le3/x$: indeed, writing $L_x=\log((x+55)/56)$, we have
\begin{equation}\label{eq:Fx}
 F'(x)=4xL_x+\frac{2x^2}{x+55}
 \le6xL_x\le\frac{3F(x)}x.
\end{equation}
Integrating over $[s,s+1]$ gives
\begin{align}
 \frac{J(s+1)}{J(s)}=2^{-1/3}\exp\left(\int_{s}^{s+1} \frac{F'(x)}{F(x)}\mathrm{d}x\right)
 &\overset{\eqref{eq:Fx}}{\le}\exp(3/s)2^{-1/3}\notag\\
 &\le\exp(1/19)2^{-1/3}
 <\frac{17}{20}.
 \label{eq:ratio}
\end{align}
 Moreover, $L_{57}=\log2<7/10$, and therefore, using $2^{-1/3}<\frac{4}{5}$, 
\begin{equation}\label{eq:first-J}
 J(57)
 <\left(2\cdot57^2\cdot\frac7{10}+5\right)
       \frac{4}{5}\,2^{-18}
 =\frac{1423}{102400}.
\end{equation}
In particular, all $J(s)$ with $s\ge57$ are less than $1$ due to~\eqref{eq:ratio}.

\medskip

\noindent\textbf{Choosing a spine and descending the target.}
Delete the bad vertices from $V_2$. Their degrees are below the guaranteed average degree, because $r>2s$ implies $1/s^2>2/r^2$. Deleting them therefore does not decrease the density. Lemma~\ref{lem:count} supplies a set $U\subseteq V_1$ of size $r-k$ whose common neighborhood $W$ among the good vertices (in $V_2$) satisfies
\begin{align}
 |W|
 &\ge(1-\theta)n_2
       \left(\frac{1-2/r^2}{b}-\frac r{n_1}\right)^{r-k}\notag\\
 &\ge(1-\theta)n_2b^{-(r-k)}
       \left(1-\frac3r\right).
 \label{eq:spine}
\end{align}
To obtain the second inequality, apply Bernoulli's inequality as in \eqref{eq:terminal}; the exponent $r-k$ is at most $r$, and the base after factoring out $b^{-1}$ lies in $[0,1]$.

Every vertex of $W$ had degree at least $(1-1/s^2)n_1/b$ into $V_1$. After deleting $U$, its degree fraction into $V_1\setminus U$ is at least
\begin{align}
 \frac{1-1/s^2}{b}-\frac r{n_1-r}
 &\ge\frac{1-1/s^2}{b}-\frac{2r}{n_1}
 \ge\frac{1-2/s^2}{b}.
 \label{eq:inherited-degree}
\end{align}
Here $n_1\ge2r$, and
\[
 \frac{2br}{n_1}\le\frac1{28r^2}\le\frac1{s^2},
\]
by \eqref{eq:n-lower} and $s\le r$. The first estimate in \eqref{eq:inherited-degree} is deliberately weaker than the direct quotient bound, so it remains valid regardless of which neighbors were removed.

The graph between $V_1\setminus U$ and $W$ contains no $K_{k,s}$ with $k$ vertices in $V_1\setminus U$: adjoining $U$ would give the forbidden $K_{r,s}$. Exchange the classes. The new graph is therefore $K_{s,k}$-free, with the larger target $s$ in its first class, density at least $(1-2/s^2)/b$, and normalized sizes
\begin{align*}
 \alpha'&=|W|b^{-k}
       \ge\beta(1-\theta)(1-3/r),\\
 \beta'&=(n_1-(r-k))b^{-s}
       \ge\alpha(1-r/n_1).
\end{align*}
Since $(1-\theta)(1-3/r)\ge1-\theta-3/r$ and $r/n_1\le4/r$, we obtain
\begin{equation}\label{eq:one-loss}
 \min\{\alpha',\beta'\}
 \ge\min\{\alpha,\beta\}(1-\lambda_s).
\end{equation}
This is the desired step $(r,s)\mapsto(s,g_b(s))$.

\medskip

\noindent\textbf{Total loss and termination.}
At each step with $s\ge57$, the next smaller target is $g_b(s)<s$. Thus the integers $s$ indexing the steps are distinct. Equations \eqref{eq:ratio} and \eqref{eq:first-J} imply
\begin{equation}\label{eq:total-loss}
 \sum\lambda_s
 <\sum_{s=57}^{\infty}J(s)
 \le\frac{20}{3}J(57)
 <\frac{1423}{15360}<\frac3{32}.
\end{equation}
For numbers in $[0,1]$, the elementary product inequality
$\prod_i(1-x_i)\ge1-\sum_i x_i$ follows by induction. Hence \eqref{eq:one-loss} gives, at every stage reached from the initial graph,
\[
 \min\{\alpha,\beta\}
 >62\left(1-\frac3{32}\right)
 =\frac{1798}{32}>56.
\]
There is no circular assumption in using \eqref{eq:maintain}. If a first stage with a normalized size below $56$ existed, all preceding steps would satisfy its hypotheses, and the displayed product estimate would rule out that first stage. The lower bound is therefore maintained throughout.

The larger target decreases strictly at every nonterminal step, and remains at least $57$ when such a step is taken. Eventually a stage with $s\le56$ is reached, where \eqref{eq:terminal} gives the contradiction. This proves the lemma.
\end{proof}

It remains to initialize the recursion from the density assumption of Theorem~\ref{thm:density}. 

\begin{proof}[Proof of Theorem~\ref{thm:density}]
Let the original classes have sizes $n_1,n_2\ge128p^{-t}$. Throughout this initialization, $n_1,n_2$ refer to these original sizes.

\medskip

\noindent\textbf{Small targets.}
Suppose first that $t<64$. Lemma~\ref{lem:count} gives an average common-neighbor count for $t$-sets in $V_1$ of at least
\begin{align*}
 n_2(p-t/n_1)^t=n_2p^t\left(1-\frac{t}{pn_1}\right)^t
 &\ge n_2p^t\left(1-\frac{t^2}{pn_1}\right)\\
 &\ge128-t^2p^{t-1}
 \ge128-t^2 2^{1-t}
 \ge128-\frac94>t,
\end{align*}
where we applied Bernoulli's inequality and used 
 $n_1\ge 128 p^{-t}$ and hence $t^2/(pn_1)\le t^2p^{t-1}/128\le9/512$, as $p\le 1/2$ and $t^22^{1-t}$ is at most $9/4$ over positive integers. 
 Thus a $t$-set has at least $t$ common neighbors, as required.

\medskip

\noindent\textbf{A slightly reduced density and a large initial spine.}
Assume $t\ge64$, and put
\begin{equation}\label{eq:initial-parameters}
 w=2^{-1/t},\qquad b=(pw)^{-1},\qquad
 s=g_b(t)=1+\lceil 3\log_b t\rceil,\qquad a=t-s.
\end{equation}
Then $b\ge2$ and $w^t=1/2$. Furthermore,
\begin{equation}\label{eq:initial-s}
 s\le2+3\log_2t\le t/3.
\end{equation}
The second inequality holds at $t=64$, and the ratio $(2+3\log_2t)/t$ decreases for $t\ge64$. In particular $a\ge1$.

Call a vertex in the first class \emph{high} if its degree into the second class is at least $pw n_2$. Choose an $a$-set in the second class uniformly. A vertex of degree $d$ is its common neighbor with probability
\[
 \frac{\binom da}{\binom{n_2}a}\le(d/n_2)^a.
\]
 Thus the total expected contribution of vertices which are not high is at most $n_1(pw)^a$. On the other hand, Lemma~\ref{lem:count} bounds the expected number of all common neighbors from below by $n_1(p-a/n_2)^a$. Subtracting the former contribution from the latter shows that some $a$-set $U\subseteq V_2$ has a set $H\subseteq V_1$ of high common neighbors with
\begin{equation}\label{eq:H}
 |H|\ge n_1\bigl[(p-a/n_2)^a-(pw)^a\bigr]
 \ge n_1p^a(1-\delta-w^a),
 \qquad \delta=\frac{a^2}{pn_2}.
\end{equation}
Bernoulli's inequality gives the last step. Its base is positive since
$pn_2\ge128\,2^{t-1}>a$. Also,
\begin{equation}\label{eq:delta}
 \delta\le\frac{t^2 2^{1-t}}{128}<\frac1{1000}
 \qquad(t\ge64).
\end{equation}

Retain all of $H$ and delete $U$ from the second class. Any $K_{t,s}$ between $H$ and $V_2\setminus U$ completes to the original $K_{t,t}$ by adjoining $U$. Normalize the residual class sizes as
\[
 \alpha=|H|b^{-s},\qquad
 \beta=(n_2-a)b^{-t}.
\]
Using $a+s=t$ and $w^t=1/2$, we obtain
\begin{align}
 \alpha
 &\ge n_1p^t\bigl(w^s-w^t-\delta w^s\bigr)
 \ge128\left(w^s-\frac12-\delta\right),
 \label{eq:alpha-init}\\
 \beta
 &\ge64-t2^{-t}>63.
 \label{eq:beta-init}
\end{align}
For \eqref{eq:beta-init}, use $n_2p^tw^t\ge64$ and
$a(pw)^t\le t2^{-t}$. To justify replacing $n_1p^t$ by its lower bound in \eqref{eq:alpha-init}, note that
\[
 w^s=\exp\left(-\frac{s\log2}{t}\right)
 \ge1-\frac st\ge\frac23.
\]
The bracket is therefore positive by \eqref{eq:delta}. In particular,
\begin{equation}\label{eq:H-rough}
 \alpha>128\left(\frac23-\frac12-\frac1{1000}\right)>20,
 \qquad |H|\ge20bt^3.
\end{equation}
The last estimate follows from $b^s\ge bt^3$.

\medskip

\noindent\textbf{The inherited density.}
Every vertex of $H$ is high in the original graph, so its degree fraction into $V_2\setminus U$ is at least
\begin{equation}\label{eq:initial-density}
 \frac{pw n_2-a}{n_2-a}
 \ge pw-\frac a{n_2-a}=\frac{1}{b}-\frac a{n_2-a}
 \ge\frac{1-2/t^2}{b}.
\end{equation}
For the last inequality, $b\le2/p$, $a\le t$, and $n_2\ge2a$ give
\[
 \frac{ba}{n_2-a}\le\frac{4t}{pn_2}
 \le\frac{t2^{1-t}}{32}<\frac2{t^2}\qquad(t\ge64).
\]
 Hence the residual graph $G[H\dot\cup (V_2\setminus U)]$ has the density required in Lemma~\ref{lem:core}.

\medskip

\noindent\textbf{Finishing the initialization.}
If $s\le56$, we need only the rough lower bound \eqref{eq:H-rough}, not $\alpha\ge62$. Lemma~\ref{lem:count} applied in the residual graph gives an average common-neighbour count for a $t$-set in $H$ of at least
\begin{align*}
 \beta\left(1-\frac2{t^2}-\frac{bt}{|H|}\right)^t
 &\ge\beta\left(1-\frac2t-\frac{bt^2}{|H|}\right)\\
 &\ge\beta(1-3/t)
 >63(1-3/64)>60>s-1.
\end{align*}
The base is positive by \eqref{eq:H-rough}, which also gives
$bt^2/|H|\le1/(20t)\le1/t$. This supplies a biclique $K_{t,s}$, which together with $U$ provides the required biclique $K_{t,t}$.

If $s\ge57$, the ceiling relation gives
\[
 t>b^{(s-2)/3}\ge2^{55/3}>2^{18}>4096.
\]
For $t\ge4096$, the decreasing ratio used in \eqref{eq:initial-s} yields
\[
 \frac st\le\frac{2+3\log_2t}{t}
 \le\frac{38}{4096}<\frac1{100}.
\]
Therefore $w^s\ge1-s/t>99/100$, and \eqref{eq:alpha-init} gives
\[
 \alpha>128\left(\frac{99}{100}-\frac12-\frac1{1000}\right)
 =62.592>62.
\]
Together with \eqref{eq:beta-init} and \eqref{eq:initial-density}, this verifies every hypothesis of Lemma~\ref{lem:core}, with $r=t$. Its $K_{t,s}$, together with $U$, gives $K_{t,t}$. Theorem~\ref{thm:density} is proved.
\end{proof}

\section{Concluding remarks and further consequences}\label{sec:closing}

The argument in Section~4 uses one graph throughout. The density parameter is allowed to be any real $p\in(0,1/2]$, and the two class sizes need not be equal. These two features give the following consequences without any additional structural argument.

\begin{corollary}[More than two colors]\label{cor:q}
For integers $q\ge2$ and $t\ge1$,
\[
 b_q(t)\le128q^t.
\]
More generally, every $q$-coloring of a complete bipartite graph with both class sizes at least $128q^t$ contains a monochromatic $K_{t,t}$.
\end{corollary}
\begin{proof}
A largest color class has density at least $1/q\le1/2$. Apply Theorem~\ref{thm:density} with $p=1/q$.
\end{proof}

\begin{corollary}[A bound for Zarankiewicz numbers]
Let $t,n_1,n_2$ be positive integers, and put
$n_*:=\min\{n_1,n_2\}$. If $n_*\ge 128\cdot 2^t$, then
\[
  z(n_1,n_2;t,t)
  <128^{1/t}n_1n_2n_*^{-1/t}.
\]
\end{corollary}

\begin{proof}
Put
\[
  p:=\left(\frac{128}{n_*}\right)^{1/t}.
\]
The assumption on $n_*$ gives $0<p\le 1/2$, and
$n_*=128p^{-t}$. By Theorem~1.1, every bipartite graph
with class sizes $n_1,n_2$ and at least $pn_1n_2$ edges
contains $K_{t,t}$. Consequently, every $K_{t,t}$-free
graph with these class sizes has fewer than
\[
  pn_1n_2=128^{1/t}n_1n_2n_*^{-1/t}
\]
edges, as required.
\end{proof}

In particular, for $n\ge 128\cdot 2^t$,
\[
  z(n,n;t,t)<128^{1/t}n^{2-1/t}.
\]
The significance here is the uniform dependence on $t$
in the dense regime; the exponent $2-1/t$ for fixed $t$
is unchanged.

\providecommand{\bysame}{\leavevmode\hbox to3em{\hrulefill}\thinspace}
\providecommand{\MR}{\relax\ifhmode\unskip\space\fi MR }
\providecommand{\MRhref}[2]{%
  \href{http://www.ams.org/mathscinet-getitem?mr=#1}{#2}
}
\providecommand{\href}[2]{#2}

\end{document}